\documentclass[uplatex,a4j,12pt,dvipdfmx,leqno]{article}

\usepackage{amsmath} 
\usepackage{amssymb} 
\usepackage{amsfonts} 
\usepackage{amsthm}
\usepackage{tikz}
\usepackage{amscd}

\theoremstyle{definition}
\newtheorem{dfn}{Definition}[subsection]
\newtheorem{dfnS}{Definition}[section]
\newtheorem{rmk}[dfn]{Remark}
\newtheorem{rmkS}[dfnS]{Remark}
\newtheorem{exm}[dfn]{Example}

\theoremstyle{plain}
\newtheorem{thm}[dfn]{Theorerm}

\newtheorem{lem}[dfn]{Lemma}

\newtheorem{prop}[dfn]{Proposition}

\newtheorem{coro}[dfn]{Corollary}

\makeatletter
  
  \@addtoreset{equation}{section}
 \makeatother

\title{On mod 2 arithmetic Dijkgraaf-Witten invariants for certain real quadratic number fields}
\author{Hikaru Hirano}
\date{}

\begin{document}

\maketitle

\begin{abstract}
Minhyong Kim introduced arithmetic Chern-Simons invariants for totally imaginary number fields as arithmetic analogues of the Chern-Simons invariants for 3-manifolds. In this paper, we extend Kim's definition for any number field, by using the modified \'etale cohomology groups and fundamental groups which take real places into account. We then show explicit formulas of mod 2 arithmetic Dijkgraaf-Witten invariants for real quadratic fields $\mathbb{Q} (\sqrt{p_1  p_2  \cdots  p_r})$, where $p_i$ is a prime number congruent to 1 mod 4, in terms of the Legendre symbols of $p_i$'s. We also show topological analogues of our formulas for 3-manifolds.

\end{abstract}

\renewcommand{\thefootnote}{\fnsymbol{footnote}}
\footnote[0]{2010 {\it Mathematics Subjct Classification.} 11R37, 11R80, 14F20, 57M27, 81T45. }

\section{Introduction}

\ \ \ \ In recent years, Minhyong Kim ([K], [CKKPY]) initiated to study arithmetic Chern-Simons theory for  number fields as an arithmetic analogy of the Dijkgraaf-Witten theory for 3-manifolds ([DW]), based on the analogies between number rings and 3-manifolds, primes and knots in arithmetic topology ([Mo]). Kim's theory is concerned with totally imaginary number fields, since  it employs some results on \'etale cohomology groups of the integer rings of totally imaginary number fields ([Ma]), which no longer hold for  number fields with real places. So it is desirable to extend Kim's theory for number fields with real places. \par

In this paper, we extend Kim's theory for  number fields with real places, by using the modified \'etale cohomology groups and the modified \'etale fundamental groups which take real primes into account, and we then compute explicitly the {\rm mod} 2 arithmetic Dijkgraaf-Witten invariants for real quadratic fields $\mathbb{Q} (\sqrt{p_1  p_2  \cdots  p_r})$, where $p_i$ is a prime number congruent to 1 mod 4, in terms of the Legendre symbols of $p_i$'s. Let us give an outline of the construction of arithmetic Chern-Simons invariants and arithmetic Dijkgraaf-Witten invariants in the following. \par

Let $K$ be a finite algebraic number field containing $n$-th roots of unity. Note that if $K$ has a real prime, $n$ must be 2. We choose a primitive $n$-th root of unity $\zeta_n$ in $K$ which induces the isomorphism $ \mathbb{Z}/ n \mathbb{Z} \cong \mu_n $.  Let ${\cal O}_K$ be the ring of integers of $K$ and let $X\ =\ {\rm Spec}\ {\cal O}_K$ be the prime spectrum of ${\cal O}_K$. Let $X_{\infty}$ denote the set of infinite primes of $K$ and we set $\overline{X} \ = \ X \sqcup X_{\infty}$. Following [B] and [AC], we can introduce a Grothendieck topology (site) ${\overline{X}}_{{\rm \acute{e}t}}$, called the Artin-Verdier site, and the topos Sh$({\overline{X}}_{{\rm \acute{e}t}})$ of abelian sheaves on ${\overline{X}}_{{\rm \acute{e}t}}$. Thus we have the modified \'etale cohomology groups ${{\rm H}}^i({\overline{X}},  F)$ for $F \in {\rm Sh}({\overline{X}}_{\rm \acute{e}t})$ and $i \geq0 $. These cohomology groups enjoy the 3-dimensional Artin-Verdier duality and we have the fundamental class isomorphism ${{\rm H}}^3({\overline{X}},  \mathbb{Z}/ n \mathbb{Z}) \cong \mathbb{Z}/ n \mathbb{Z}$ that depends on $\zeta_n$. We may also have the category of finite \'etale coverings over $\overline{X}$, which is proven to be a Galois category. Thus we have the modified \'etale fundamental group $\pi_1(\overline{X})$. Now, let $A$ be a finite group and let $c \in {{\rm H}}^3 (A,  \mathbb{Z}/ n \mathbb{Z})$. We set ${\cal M}(\overline{X},A) ={\rm Hom}_c(\pi_1(\overline{X}),A) / A$ be the set of conjugate classes of all continuous homomorphisms $\pi_1(\overline{X}) \rightarrow A$. Then, for $\rho \in {\cal M}(\overline{X},  A)$, the  {\it arithmetic Chern-Simons invariant} $CS_c(\rho)$ of $\rho$ associated to $c$ is defined by the image of $c$ under the composition of the maps
\begin{equation}
{{\rm H}}^3( A,  \mathbb{Z}/n \mathbb{Z}) \xrightarrow{\rho^{\ast}} {{\rm H}}^3( \pi_1(\overline{X}),  \mathbb{Z}/n \mathbb{Z}) \xrightarrow{j_3} {{\rm H}}^3( \overline{X},  \mathbb{Z}/n \mathbb{Z}) \cong \mathbb{Z}/n \mathbb{Z}, \nonumber
\end{equation}
where $j_3$ is the edge homomorphism in the modified Hochschild-Serre spectral sequence ${{\rm H}}^p( \pi_1(\overline{X}),  {{\rm H}}^q( \tilde{\overline{X}},  \mathbb{Z}/n \mathbb{Z} )) \Rightarrow {{\rm H}}^{p+q} ( \overline{X},  \mathbb{Z}/n \mathbb{Z} )$ (see the section 2 for $\tilde{\overline{X}}$).
The  {\it arithmetic Dijkgraaf-Witten invariant} of $\overline{X}$  associated to $c$ is then defined by
\begin{equation}
\displaystyle{Z_c (\overline{X}) = \sum_{\rho \in  {\cal M}(\overline{X},A)} {\mathrm exp} \left( \frac{2 \pi i}{n} CS_c(\rho) \right)   }. \nonumber
\end{equation}
\par

The basic problem may be to compute concretely ${CS}_c(\rho)$ and $Z_c(\overline{X})$. The papers [CKKPPY], [CKKPPY] and [BCGKPT] are concerned with this problem for the cases where $K$ is totally imaginary and $c$ is some specific cocycle.
 In this paper, we consider the case where $K$ is the real quadratic field $\mathbb{Q} (\sqrt{p_1  p_2  \cdots  p_r})$,  $p_i$ being a prime number congruent to 1 mod 4, $A=\mathbb{Z}/2\mathbb{Z}$ and $c$ is the non-trivial cocycle in ${{\rm H}^3}(A,  \mathbb{Z}/2\mathbb{Z})=\mathbb{Z}/2\mathbb{Z}$. For this, we firstly generalize a result in [AC] and [BCGKPT], which describes ${CS}_c(\rho)$ by the Artin symbol in unramified class field theory, for any number fields. Then, by using Gauss' genus theory, we compute explicitly ${CS}_c(\rho)$ and ${Z}_c(\overline{X})$ for the above case, in terms of the Legendre symbols among $p_i$'s. \par Following the analogies in arithmetic topology, in Appendix, we show a topological counterpart of our main result in the context of Dijkgraaf-Witten theory for 3-manifolds. \par
Here are the contents of this paper. In the section 2.1, notations being as above, we introduce the Artin-Verdier site ${\overline{X}}_{{\rm \acute{e}t}}$ and the category ${{\rm FEt}}_{\overline{X}}$ of finite \'etale coverings over $\overline{X}$. We show that ${{\rm FEt}}_{\overline{X}}$ is a Galois category and define the modified \'etale fundamental group $\pi_1(\overline{X})$ as the automorphism group of the fiber functor of ${{\rm FEt}}_{\overline{X}}$. In the section 2.2, we introduce the topos ${\rm Sh}({\overline{X}}_{{\rm \acute{e}t}})$ of abelian sheaves on 
${\overline{X}}_{{\rm \acute{e}t}}$ and define the modified \'etale cohomology groups ${{\rm H}}^i(\overline{X},F)$ for $F \in {\rm Sh}({\overline{X}}_{{\rm \acute{e}t}})$ and $i \geq 0$. We also show the modified Hochschild-Serre spectral sequence. In section 3, by using the materials prepared in the sections 2.1 and 2.2. In the section 4, we firstly extend a result in [AC] for $\overline{X}$. Then, we compute explicitly mod 2 $CS_{c}(\rho)$ and $Z_c(\overline{X})$ for $K=\mathbb{Q} (\sqrt{p_1  p_2  \cdots  p_r})$, $p_i \equiv 1\  {\rm mod} \ 4 $. In Appendix, we show a topological counterpart of our formulas in the section 4, in the context of Dijkgraaf-Witten theory for 3-manifolds.  \par
The contents of this paper were announced by the author at the workshop ``Low dimensional topology and number theory XI" held in Osaka University on March of 2019. During the preparation of this paper after that, we found the paper [LP] which also studies the arithmetic Chern-Simons theory for number fields with real primes. They use compactly supported \'etale cohomology groups. \\\\ \noindent
{\it Notation.}\ As usual, we denote by $\mathbb{Z}$, $\mathbb{Q}$, $\mathbb{R}$, and $\mathbb{C}$ the ring of rational integers, the field of rational numbers, the field of real numbers and the field of complex numbers,  respectively. For a commutive ring $R$, we denote by $R^{\times}$ the group of units in $R$.\\
For a number field $K$, we denote by ${\cal O}_K$ the ring of integers of $K$.  We denote by $I_K$ the group of fractional ideals of $K$, and ${\rm N} \mathfrak{a} $ denotes the norm of $\mathfrak{a} \in I_K$. We denote by ${\rm Cl}_K$, (resp. ${\rm Cl}^{+}_K$ ) the ideal class group, (resp. the narrow ideal class group) of $K$. \\\\ \noindent
{\bf{\it Acknowledgement.}}\ The author woulld like to thank his supervisor Masanori Morishita for suggesting the problem studied in this paper. He is also thankful to Junhyong Kim for discussion and to Yuji Terashima for communication.

\section{The modified \'etale cohomology groups for a number ring}

\ \ \ In the section 2.1, following [AC] and [B], we recall the Artin-Verdier site for a number field. We then define the modified \'etale fundamental group of the ring of integers, taking the infinite primes into account. In the section 2.2, we define the modified \'etale cohomology groups following [AC] and [B], and then we show the Hochschild-Serre spectral sequence.
\subsection{The Artin-Verdier site and the modified \'etale fundamental group}
\ \ \ Let $K$ be a finite algebraic number field and let $X={\rm Spec}\ {\cal O}_k$ be the prime spectrum of the ring ${\cal O}_k$ of integers of  $K$. Let $X_{\infty}$ be the set of infinite primes, namely real primes and pairs of conjugate complex primes of $K$, and we set $\overline{X} \ = \ X \sqcup X_{\infty}$. A real prime of a scheme $Y$ \'etale over $X$ is defined by a point $y:{\rm Spec}\ \mathbb{C} \rightarrow Y$ which factors through ${\rm Spec}\ \mathbb{R}$. A complex prime of $Y$ is defined by a pair of complex conjugate points $y,\overline{y}:{\rm Spec}\ \mathbb{C} \rightarrow Y$ such that $y \neq \overline{y}$. An infinite prime of $Y$ is a real prime or a complex prime of $Y$. Let $Y_{\infty}$ be the set of infinite primes of $Y$. Note that an  \'etale morphism $f:Y\rightarrow X$ induces $f_{\infty}:Y_{\infty} \rightarrow X_{\infty}$. We say that $f_{\infty}$ is unramified at $y_{\infty} \in Y_{\infty}$ if $y_{\infty}$ is a real prime or if $y_{\infty}$ and $f_{\infty}(y_{\infty})$ are complex primes. For  the Grothendieck topologies, we refer to [Ar], [T].
\begin{dfn}[{ [AC; Definition 2.1], [B; Proposition 1.2]}]
 The {\it Artin-Verdier site} of $\overline{X}$, denoted by ${\overline{X}}_{\rm \acute{e}t}$, is the Grothendieck topology consisting of the category ${\rm Cat}({\overline{X}}_{{\rm \acute{e}t}})$ and a set ${\rm Cov}({\overline{X}}_{{\rm \acute{e}t}})$ of coverings defined as follows.\\
・An object in ${\rm Cat}({\overline{X}}_{{\rm \acute{e}t}})$ is a pair $(Y,M)$, where $f:Y \rightarrow X$ is a scheme \'etale over X and $M \subset Y_{\infty} $ such that $f_{\infty}| : M \rightarrow X_{\infty}$ is unramified. A morphism $\varphi : (Y_1,M_1) \rightarrow (Y_2,M_2)$ in ${\rm Cat}({\overline{X}}_{{\rm \acute{e}t}})$ is a morphism of shemes $\varphi : y_1 \rightarrow Y_2$ over X such that the induced map $\varphi_{\infty} : {(Y_1)}_{\infty} \rightarrow {Y_2}_{\infty}$ satisfies $\varphi_{\infty} (M_1) \subset M_2$.\\
・A covering in ${\rm Cov}({\overline{X}}_{{\rm \acute{e}t}})$ is a family of morphisms $ \{ \varphi_i : (Y_i,M_i) \rightarrow (Z,N) \}_{i \in I}$ in ${\rm Cat}({\overline{X}}_{{\rm \acute{e}t}})$ which satisfies $\underset{i}{\cup} \varphi_i (Y_i) = Z $ and $\underset{i}{\cup} \varphi_i (M_i) = N.$
\end{dfn}

\begin{rmk}
For  morphisms $\varphi_i : (Y_i,M_i) \rightarrow (Z,N)$ $(i=1,2)$ in ${\rm Cat}({\overline{X}}_{{\rm \acute{e}t}})$, the fiber product of them is defined by  $(Y_1 \underset{Z}{\times} Y_2,M_3)$, where $Y_1 \underset{Z}{\times} Y_2$ is the fiber product in the category of schemes and $M_3$ is the set consisting of points of $(Y_1 \underset{Z}{\times} Y_2)_{\infty}$ whose images are in $M_i$ under the projections $(Y_1 \underset{Z}{\times} Y_2)_{\infty} \rightarrow {Y_i}_{\infty}$ for $i=1,2$. We can check easily $M_3$ is isomorphic to $M_1 \underset{N}{\times} M_2$ in the category of sets.
\end{rmk}

Next, we introduce a Galois category to define the modified \'etale fundamental group.\par
We say $(Y,M)\in {\rm Cat}({\overline{X}}_{{\rm \acute{e}t}})$ is {\it finite \'etale} if $Y \rightarrow X$ is a finite \'etale morphism of schemes over $X$ and $M=Y_{\infty}$.  Then, we denote by ${{{\rm FEt}}}_{\overline{X}}$ the full subcategory of ${\overline{X}}_{{\rm \acute{e}t}}$ whose objects are finite \'etale, and denote by FSets the category of finite sets.\par
In the following, we often abbreviate $(Y,Y_{\infty})$ to $\overline{Y}$ for a scheme $Y$ \'etale over $X$. Let $\overline{K}$ be an algebraic closure of $K$ and let $\tilde{\eta}:{\rm Spec}\ \overline{K}\ \rightarrow X$ be a geometric point. Then we have a functor
\begin{equation}
F_{\tilde{\eta}} : {{\rm FEt}}_{\overline{X}} \rightarrow {\rm FSets}; \  \overline{Y} \mapsto {\rm Hom}_{X} (\tilde{\eta},  Y). \nonumber
\end{equation}

\begin{prop}
${{\rm FEt}}_{\overline{X}}$ is a Galois category with a fiber functor $F_{\tilde{\eta}}$.
\end{prop}
\begin{proof}
We check the six axioms (G1)～(G6) of Galois categories for ${{\rm FEt}}_{\overline{X}}$ and $F_{\tilde{\eta}}$ ([SGA I ; V.4]). The fact that the category of schemes finite \'etale over $X$, denoted by ${{\rm FEt}}_{X}$, is a Galois category with a fiber functor $F'_{\tilde{\eta}} : {{\rm FEt}}_{X} \rightarrow {\rm FSets} $ $Y \mapsto {\rm Hom}_X (\tilde{\eta},  Y)$  is well-known ([SGA I ; V.7]). So we may admit the axioms (G1)～(G6) for ${{\rm FEt}}_{X}$ and $F'_{\tilde{\eta}}$. Let us verify (G1)～(G6) for ${{\rm FEt}}_{\overline{X}}$ and $F_{\tilde{\eta}}$.\\
(G1): ${\rm FEt}_{\overline{X}}$ has a final object $(id:X \rightarrow X,  X_{\infty})$. For $\overline{Y_i} \in {\rm FEt}_{\overline{X}} \  (i=1,2,\cdots,m)$, There exist ${\displaystyle \prod_{i} Y_i \in {\rm FEt}_{X}}$ and one can see ${\displaystyle \prod_{i} \overline{Y_i} = \overline{\prod_{i} Y_i}}$ by checking the universal property of fiber products.\\
(G2): ${\rm FEt}_{\overline{X}}$ has an initial object $({\rm Spec}\ 0,({\rm Spec}\ 0)_{\infty})=(\emptyset,\emptyset)$. By similar way from (G1), one can check that ${\rm FEt}_{\overline{X}}$ has finite direct sum. For $\overline{Y} \in  \ {\rm FEt}_{\overline{X}}$ and a finite subgroup $G \subset {\rm Aut}_{\overline{X}} (\overline{Y})$, we can see ${\rm Aut}_{\overline{X}} (\overline{Y})={\rm Aut}_X (Y)$ by the definition of morphisms of ${\rm Cat}({\overline{X}}_{{\rm \acute{e}t}})$. So there is a quotient of $Y \rightarrow X \in  \ {\rm FEt}_{X}$ by $G \subset {\rm Aut}_{X} (Y)$ and then one can check $\overline{Y}/G=\overline{Y/G}.$\\
(G3): For any morphism $\overline{Y_1} \rightarrow \overline{Y_2}$ in ${\rm FEt}_{\overline{X}}$, $Y_1 \rightarrow Y_2$ factors\\ $Y_1 \xrightarrow{f} Y' \xrightarrow{g} Y' \bigsqcup Y'' \cong Y_2$ in ${\rm FEt}_{X}$, where $f$ is a strict epimorphism and $g$ is a monomorphism. This sequence induces $\overline{Y_1} \xrightarrow{f} \overline{Y'} \xrightarrow{g} \overline{Y'} \bigsqcup \overline{Y''} \cong \overline{Y_2}$.\\
(G4) and (G5) are obvious because $F_{\tilde{\eta}} (\overline{Y})=F'_{\tilde{\eta}}$ and $F'_{\eta}$ is a fiber functor of ${\rm FEt}_{X}$.\\
(G6): If $\overline{Y_1}\rightarrow \overline{Y_2}$ is isomorphism, then $F_{\tilde{\eta}} (\overline{Y_1})=F'_{\tilde{\eta}} (Y_1) \rightarrow F'_{\tilde{\eta}} (Y_2)=F_{\tilde{\eta}} (\overline{Y_2})$ is isomorphism. Conversely, if $F_{\tilde{\eta}} (\overline{Y_1}) \rightarrow F_{\tilde{\eta}} (\overline{Y_2})$ is an isomorphism, then $Y_1 \rightarrow Y_2$ is an isomorphism and that induces $\overline{Y_1} \cong \overline{Y_2}$.
\end{proof}
Now we move to define the modified \'etale fundamental group.
\begin{dfn}
The {\it modified \'etale fundamental group} with geometric basepoint $\tilde{\eta}$, denoted by $\pi_{1} (\overline{X})=\pi_1(\overline{X},\tilde{\eta})$, is defined by the fundamental group of the Galois category ${\rm FEt}_{\overline{X}}$ associated to the fiber functor $F_{\tilde{\eta}}$, namely the group of automorphisms of $F_{\tilde{\eta}}.$
\end{dfn}
By the main theorem of Galois categories, we have the following.
\begin{thm}
There is an equivalence of categories between ${\rm FEt}_{\overline{X}}$ and the category of finite discrete sets equipped with continuous left action by $\pi_{1} (\overline{X})$.
\end{thm}

Next, to describe $\pi_{1} (\overline{X})$ more explicitly, we see which object is Galois in the Galois category ${\rm FEt}_{\overline{X}}$. By definition of a connected object and a Galois object in a Galois category, one can see that $\overline{Y} \in  \ {\rm FEt}_{\overline{X}}$ is connected in ${\rm FEt}_{\overline{X}}$ iff $Y \rightarrow X$ is connected in ${\rm FEt}_{X}$, and that a connected object $\overline{Y}$ is Galois in ${\rm FEt}_{\overline{X}}$ iff ${\rm Aut}_{\overline{X}} (\overline{Y}) = {\rm Aut}_{X }(Y) \rightarrow F'_{\tilde{\eta}} (Y) = F_{\tilde{\eta}} (\overline{Y}) $ is bijective, i.e, Galois in ${\rm FEt}_{X}$.  So we have the following Proposition. Let $\tilde{K}$ (resp. ${\tilde{K}}^{ab}$) be the maximal Galois (resp. abelian) extension of $K$ which is unramified over all finite and infinite primes.

\begin{prop}
Notations being as above, we have
\begin{equation}
\pi_{1}(\overline{X})={\rm Gal}({\tilde{K}}/K). \nonumber
\end{equation}
For the abelianization $\pi^{ab}_{1}(\overline{X})$ of $\pi_{1}(\overline{X})$, we have
\begin{equation}
\pi^{ab}_{1}(\overline{X})={\rm Gal}({\tilde{K}}^{ab}/K) \cong {\rm Cl}_K. \nonumber
\end{equation}
\end{prop}

\begin{proof}
The first assertion follows from the definition of $\pi_{1}(\overline{X})$ given above. The second assertion follows from the Artin reciprocity isomorphism 
\begin{center}
${\rm Cl}_K \stackrel{\sim}{\rightarrow} {\rm Gal}({\tilde{K}}^{ab}/K)$ ; $ [\mathfrak{a}] \mapsto \left( \frac{{\tilde{K}}^{ab}/K}{\mathfrak{a}} \right)$.
\end{center}
\end{proof}

\subsection{The Artin-Verdier topos and the modified \'etale cohomology groups}
\ \ \ Let Sh$({\overline{X}}_{{\rm \acute{e}t}})$ be the category of abelian sheaves on the site ${\overline{X}}_{{\rm \acute{e}t}}$, called the {\it Artin-Verdier topos}. Firstly, we recall the decomposition lemma for Sh$({\overline{X}}_{{\rm \acute{e}t}})$ following to [AC] and [B]. We fix an algebraic closure $\overline{K}$ of $K$. For $x\in X_{\infty}$, we fix an extension $\overline{x}$ of $x$ to $\overline{K}$ and denote by $I_{\overline{x}}$ the inertia group of $\overline{x}$.  We have $I_{\overline{x}} \cong \mathbb{Z}/2\mathbb{Z}$  for a real prime $x$ and $I_{\overline{x}}$ is trivial for a complex prime $x$. Let $\eta : {\rm Spec}\ K \rightarrow X$ be the generic point. Then, for $F \in {\rm Sh}({\overline{X}}_{{\rm \acute{e}t}})$, we can regard $\eta^{*} F=F_{\eta}$ as a Gal$(\overline{K}/K)$-module and $I_{\overline{x}} \subset {\rm Gal}(\tilde{K}/K)$ acts on $\eta^{*} F$. We define the site, denoted by $TX_{\infty}$, as follows. An object in $TX_{\infty}$ is a pair $(M,m)$ where $M$ is a finite set and $m:M\rightarrow X_{\infty}$ is a map.  A morphism $(M_1,m_1) \rightarrow (M_2,m_2)$ in $TX_{\infty}$ is a map $f : M_1 \rightarrow M_2$ such that $m_2=f \circ m_1$.  A covering in $TX_{\infty}$ is a family of morphisms $\{ \varphi_i :(M_i,m_i) \rightarrow (M,m) \}_{i \in I}$ in $TX_{\infty}$ such that $m_i$ is surjective and $M=\underset{i}{\cup} \varphi(M_i)$. Then, we can easily identify a sheaf $G$ on $TX_{\infty}$ with a family of abelian groups $ \{G_{x} \}_{x \in {X_{\infty}}} $. We define the maps of sites $p : {\overline{X}}_{{\rm \acute{e}t}} \rightarrow TX_{\infty}$ and $q : {\overline{X}}_{{\rm \acute{e}t}} \rightarrow X_{{\rm \acute{e}t}}$ by the forgetful functors. Then we have the following functors
$${\rm Sh} (TX_{\infty}) \overset{p_{*}}{\underset{p^{*}}{\leftrightarrows}}  {\rm Sh} (\overline{X}_{{\rm \acute{e}t}})  \overset{q_{*}}{\underset{q^{*}}{\leftrightarrows}}  {\rm Sh}(X_{{\rm \acute{e}t}}).$$
Next, we define the category   ${\rm Sh}({\overline{X}}_{{\rm \acute{e}t}})'$ as follows. An object in ${\rm Sh}({\overline{X}}_{{\rm \acute{e}t}})'$ is a triple $(\{G_x\}_{x \in X_{\infty}} , F,  \{ \sigma_{x} : G_{x} \rightarrow (\eta^{*} F)^{I_{x}} \}_{x \in X_{\infty}})$, where $\{G_x\}_{x \in X_{\infty}} \in  {\rm Sh} (TX_{\infty})$, $F \in  {\rm Sh}(X_{{\rm \acute{e}t}})$
 and $\{ \sigma_{x} : G_{x} \rightarrow (\eta^{*} F)^{I_{x}} \}_{x \in X_{\infty}} $ is a family of homomorphisms of abelian groups. A morphism $(\{ G_{x} \},  F, \{ \sigma_{x} \} ) \rightarrow (\{ G'_{x} \},  F', \{ \sigma'_{x} \} )$ is a pair of morphisms $ \{ G_{x} \}  \rightarrow \{ G'_{x} \} $,  $F \rightarrow F'$ such that the induced diagram
\begin{center}
\[
  \begin{CD}
       G_{x}  @>{\scriptstyle{\sigma_{x}}}>>  (\eta^{*} F)^{I_{x}}     \\
         @VV{p_1}V  @VV{p_2}V      \\
       G'_{x} @>{\scriptstyle{\sigma'_{x}}}>>  (\eta^{*} F')^{I_{x}},
  \end{CD}
\]
\end{center}
%%\begin{center}
%%\begin{tikzpicture}[auto]
%%\node (a) at (0, 2.2) {$G_{x}$}; \node (x) at (2.2, 2.2) {$(\eta^{*} F)^{I_{x}} $};
%%\node (b) at (0, 0) {$G'_{x}$};   \node (y) at (2.2, 0) {$(\eta^{*} F')^{I_{x}} $};
%%\draw[->] (a) to node {$\scriptstyle{\sigma_{x}}$} (x);
%%\draw[->] (x) to node {} (y);
%%\draw[->] (a) to node[swap] {} (b);
%%\draw[->] (b) to node[swap] {$\scriptstyle{\sigma'_{x}}$} (y);
%%\end{tikzpicture}
%%\end{center}
is commutative for each $x \in X_{\infty}$.\par
Now we describe the statement of the decomposition lemma for Sh$({\overline{X}}_{{\rm \acute{e}t}})$.
\begin{lem}[{{\rm [AC: Proposition 2.3], [B; Proposition 1.2]}}]  There is an equivalence of categories given by the following functors.
$${\rm Sh} ({\overline{X}}_{{\rm \acute{e}t}}) \overset{\Phi}{\underset{\Psi}{\leftrightarrows}}  {\rm Sh} ({\overline{X}}_{{\rm \acute{e}t}})',$$ where $\Phi$ and $\Psi$ are defined by 
$$\Phi : S \mapsto (q_{*} S,  p_{*} S,  p_{*} S \rightarrow p_{*} q^{*} q_{*} S ),\ \Psi : (\{ G_{x} \},  F, \{ \sigma_{x} \}) \mapsto q^{*} F \times_{p^{*} q_{*} q^{*} F} p^{*} \{ G_{x} \} .$$
\end{lem}
\begin{proof}
To apply the result of [Ar; Proposition 2.4], we check the following (1), (2), (3) and (4).\\
(1) $q_{*}$(resp. $p_{*}$) is left adjoint to $q^{*}$ (resp. $p^{*}$).\\
(2) $q_{*}$,  $p_{*}$ are exact.\\
(3) $p^{*}$,  $q^{*}$ are fully faithful.\\
(4) For any $S \in {\rm Sh} ({\overline{X}}_{{\rm \acute{e}t}})$,  $q_{*} S = 0$ holds iff there exists $G \in {\rm Sh} (TX_{\infty})$ such that $S=p^{*} G$.\\
For (1), (3) and (4), see [Z; Proposition 1.3.3]. (2) follows from the fact that ${\overline{X}}_{{\rm \acute{e}t}}$, $X_{{\rm \acute{e}t}}$ and $TX_{\infty}$ have a final object and finite fiber products,  and  $p$, $q$  preserve them.
\end{proof}

\begin{rmk}
(1) Through the category equivalence, we can identify $p_{*}$ (resp. $p^{*}$,  $q_{*}$,  $q^{*}$ ) with the following functors$\psi_{*}$ (resp. $\psi^{*}$,  $\phi^{*}$,  $\phi_{*}$).
\begin{center}
$\phi^{*} (\{ G_{x} \},  F,  \{ \sigma_{x} \} ) = F $,  $\phi_{*} F = (\{ (\eta^{*} F)^{I_{x}}  \},  F,  \{ {\mathrm id} \}  )$ \\ $\psi^{*} (\{ G_{x} \},  F,  \{ \sigma_{x} \} ) = \{ G_{x} \} $,  $\psi_{*} \{ G_{x} \} = (\{ G_{x} \},  0,  \{ 0 \}).$
\end{center}
(2) If we denote by ${\underline{A}}_{{\overline{X}}_{{\rm \acute{e}t}}}$ the constant sheaf on ${\overline{X}}_{{\rm \acute{e}t}}$ associated to an abelian group $A$, then one can see ${\underline{A}}_{{\overline{X}}_{{\rm \acute{e}t}}} = \phi_{*} ({\underline{A}}_{X_{{\rm \acute{e}t}}})$. In the following, if there is no confusion, we will abbreviate ${\underline{A}}_{{\overline{X}}_{{\rm \acute{e}t}}}$ to $A$.\\
(3) For $S = (\{ G_{x} \},  F,  \{ \sigma_{x} \} ) \in {\rm Ob Sh}({\overline{X}}_{{\rm \acute{e}t}}) $,  the section of $S$ at $(Y,M) \in {\overline{X}}_{{\rm \acute{e}t}} $, $\Gamma ((Y,M),S)$, is given by $F(Y) \times_{\eta^{*} F} G_{x_{1}} \times_{\eta^{*} F} G_{x_{2}} \times_{\eta^{*} F} \cdots \times_{\eta^{*} F} G_{x_{r}}$, where $ \{ x_{1}, x_{2},  \cdots,  x_{r} \}$ is the image of $M$ by $Y_{\infty} \rightarrow X_{\infty}$.
\end{rmk}

\begin{dfn}
For $S \in {\rm Sh} ({\overline{X}}_{{\rm \acute{e}t}})$, the cohomology group ${\rm H}^i(\overline{X},  S)$ is called the  $i$-th {\it modified \'etale cohomology group} of $\overline{X}$ with values in $S$.
\end{dfn}
When $S$ is the constant sheaf $\mathbb{Z}/ n\mathbb{Z}$,  ${\rm H}^i(\overline{X},  S)$ is calculated in [B; Proposition 2.13] and [AC; Corollary 2.15].  We firstly recall the Artin-Verdier duality.
%% using the Artin-Veridier duality ([B; Proposition 5.1])
\begin{prop}[{The Artin-Verdier Duality ([B; Theorem 5.1])}]
Let $F$ be a constructible sheaf on $X={\rm Spec}\ {\cal O}_K$.  We fix an algebraic closure $\overline{K}$ of $K$, and for $x \in X_{\infty}$, we fix an extension $\overline{x}$ of $x$ to $\overline{K}$. Let $\eta : {\rm Spec}\ K \rightarrow X$ be the generic point. Let $G_{m,X}$ be the \'etale sheaf of units on $X$ Then we have the followings.\\\\
{\rm (a)} ${\rm H}^i({\overline{X}},\phi_{*} F) = {\rm Ext}^i_{{\overline{X}}}(\phi_{*} F, \phi_{*} G_{m,X})=0$ for $i>3.$\\
{\rm (b)} The Yoneda pairing $${\rm H}^i({\overline{X}},\phi_{*} F) \times {\rm Ext}^{3-i}_{{\overline{X}}}(\phi_{*} F, \phi_{*} G_{m,X}) \rightarrow {\rm H}^3({\overline{X}},\phi_{*} G_{m,X}) \cong \mathbb{Q}/\mathbb{Z} $$ is a perfect duality of finite groups for $i\geq2$.\\
{\rm (c)}  If, for any $x \in X_{\infty}$, the inertia group $I_{\overline{x}}$ of $\overline{x}$  acts trivially on the ${\rm Gal}(\overline{K}/K)\mbox{-module}$ $\eta^{*}F=F_{\eta}$, then the pairing in {\rm (b)} is perfect for any $i\geq0$.
\end{prop}
In this paper, we apply the Artin-Verdier Duality for the constant sheaf $F=\mathbb{Z}/n \mathbb{Z}$ on $X$ to obtain the following Proposition 2.2.5.
We denote by $\mu_n(K)$ the group of $n$-th roots of unity in $K$ and we define the groups $Z_1$ and $B_1$ by $Z_{1}=\{ (a, \mathfrak{a}) \in K^{\times} \oplus I_K | {(a)}^{-1} = {\mathfrak{a}}^{n}  \}$,  $B_{1}=\{ (b^n,  {(b)}^{-1}) \in K^{\times} \oplus I_K | b \in K^{\times} \} $.
\begin{prop}[{[B; Proposition 2.13], [AC; Corollary 2.15]}] We have
$${\rm Ext}^{i}_{\overline{X}} (\mathbb{Z} / n \mathbb{Z},  \phi_{*} G_{m,X}) = 
\begin{cases} 
\mu_{n}($K$) & (i=0)\\
Z_{1}/B_{1} & (i=1)\\
{\rm Cl}_K / n & (i=2)\\
\mathbb{Z} / n \mathbb{Z} & (i=3)\\
0 & (i>3),\\
\end{cases}$$
where $G_{m,X}$ is the \'etale sheaf of units on $X$. Then we have, by the Artin-Verdier duality,
$${\rm H}^{i} (\overline{X}, \mathbb{Z} / n \mathbb{Z})= 
\begin{cases} 
\mathbb{Z} / n \mathbb{Z} & (i=0)\\
({\rm Cl}_K / n)^{\sim} & (i=1)\\
(Z_{1}/B_{1})^{\sim} & (i=2)\\
(\mu_{n}($K$))^{\sim} & (i=3)\\
0 & (i>3),\\
\end{cases}$$
where $(-)^{\sim}$ is given by ${\rm Hom}(-,\mathbb{Q}/\mathbb{Z})$.
\end{prop}

\begin{rmk}
 By Theorem 2.1.6, for a continuous and surjective homomorphism $\rho : \pi_1(\overline{X}) \rightarrow \mathbb{Z} / n\mathbb{Z}$, there is a corresponding Galois object $\overline{Y} \rightarrow \overline{X}$ ($Y={\rm Spec}\ {\cal O}_L$) whose Galois group is  $\mathbb{Z} / n \mathbb{Z}$. Since $L$ is an  cyclic extension of degree $n$ unramified at all finite and infinite primes, there exists $v \in K^{\times}$ such that $L=K(v^{\frac{1}{n}})$ and there exists $\mathfrak{a} \in I_K$ which satisfies $ {\mathfrak{a} }^n = { (v)^{-1} } $. By the definition of $L$, there is an isomorphism $\overline{\rho} : {\rm Gal}(L/K) \stackrel{\sim}{\rightarrow} \mathbb{Z}/ n \mathbb{Z}$ by the Galois correspondence. Then we can identify $\rho  : \pi_1(\overline{X}) \rightarrow \mathbb{Z} / n\mathbb{Z}$ with the restriction map ${\rm Gal}(\tilde{K}/K) \rightarrow {\rm Gal}(L/K)$. Therefore, when we identify $\rho :{\pi_1}^{ab}(\overline{X}) \rightarrow {\rm Gal}(L/K)$ with $\rho' : {\rm Cl}_K \rightarrow {\rm Gal}(L/K)$ by Proposition 2.1.6,  we have $\left( \frac{L/K}{ \mathfrak{a} } \right) = \rho' ( [ \mathfrak{a}])$ for any $[ \mathfrak{a}] \in {\rm Cl}_K$.

\end{rmk}

Now we move to state the extension of Hochschild-Serre spectral sequence.
\begin{thm}
Let $\overline{Y} \rightarrow \overline{X}$ be a Galois object in ${\rm FEt}_{\overline{X}}$. Then for any $S \in {\rm Sh} ({\overline{X}}_{{\rm \acute{e}t}})$, there is a cohomological spectral sequence

\begin{equation}
{\rm H}^p( {\rm Gal} (\overline{Y} / \overline{X} ),  {\rm H}^q( \overline{Y},  S| \overline{Y} )) \Rightarrow {\rm H}^{p+q} ( \overline{X},  S ). \nonumber
\end{equation}
\end{thm}

\begin{proof}
Let ${\rm Gal} (\overline{Y} / \overline{X})$-mod denote the category of ${\rm Gal} (\overline{Y} / \overline{X}) \mbox{-}$modules. We consider the following functors
$$ \begin{array}{c}
F_1 : {\rm Sh} ({\overline{X}}_{{\rm \acute{e}t}}) \rightarrow {\rm Gal} (\overline{Y} / \overline{X}) \mbox{-mod},  S \mapsto S(\overline{Y})\\
F_2 : {\rm Gal} (\overline{Y} / \overline{X})\mbox{-mod} \rightarrow {\rm Ab},  M \mapsto M^{{\rm Gal}(\overline{Y}/\overline{X}) }
\end{array} $$
where the action of $G={\rm Gal} (\overline{Y} / \overline{X})$ on $S(\overline{Y})$ is defined by $\sigma . x =S(\sigma) (x) $ for $x \in S(\overline{Y})$ and $\sigma \in G$. Just like [Mi; Remark5.4] and [Mi; Proposition 1.4], we can easily check $(F_2 \circ F_1) (S) = S(\overline{Y})^{G} = S(\overline{X})$. To use the Grothendieck spectral sequence,  we check $F_1$ takes any injective object $I$ to a $F_2$-acyclic object. By replacing $Y$ and $X$ with $\overline{Y}$ and $\overline{X}$ in the argument of [Mi; example2.6], one can see 
${\rm H}^{i} (G, I(\overline{Y}))={\check{{\rm H}}}^{i} (\overline{Y} / \overline{X},  I )$. Since $I$ is injective,  ${\check{{\rm H}}}^{i} (\overline{Y} / \overline{X},  I )=0$ by the definition of \u{C}ech cohomologies.
\end{proof}

Let $(\overline{Y_{i}} \rightarrow \overline{X}, \overline{Y_{i}} \rightarrow \overline{Y_{j} })$ be the inverse system of finite Galois coverings over $\overline{X}$ and let $\tilde{ \overline{X} }$ be $\mathop{\varprojlim}\limits_{i} Y_i$. 
By ${\rm H}^{p}(\tilde{X},  \mathbb{Z}/ n \mathbb{Z}) = \mathop{\varprojlim}\limits_{i} {\rm H}^{p} (Y_{i},  \mathbb{Z} / n \mathbb{Z})  $ and the local cohomology sequence ([B; Proposition 1.4]), we have ${\rm H}^{p}(\tilde{\overline{X}},  \mathbb{Z}/ n \mathbb{Z}) = \mathop{\varprojlim}\limits_{i} {\rm H}^{p} (\overline{Y_{i}},  \mathbb{Z} / n \mathbb{Z})  $. So on passing to the inverse limit, we obtain the following.
\begin{coro}
There is a cohomological spectral sequence
\begin{equation}
{\rm H}^p( \pi_1(\overline{X}),  {\rm H}^q( \tilde{\overline{X}},  \mathbb{Z}/n \mathbb{Z} )) \Rightarrow {\rm H}^{p+q} ( \overline{X},  \mathbb{Z}/n \mathbb{Z} ). \nonumber
\end{equation}

\end{coro}

\section{Arithmetic Dijkgraaf-Witten invariants for a number ring}
\ \ \ In this section, we introduce  arithmetic Chern-Simons invariant for a number field, by using the modified \'etale cohomology groups in the section 2. Let $X={\rm Spec} \ {\cal O}_{K} $, the prime spectrum of the ring of integers in a number field $K$, which includes $n$-th roots of unity. We choose a primitive $n$-th root of unity $\zeta_n$ in $K$ which induces the isomorphism $ \mathbb{Z}/ n \mathbb{Z} \cong \mu_n $. Let $A$ be a finite group and let $c \in {{\rm H}}^3 (A,  \mathbb{Z}/ n \mathbb{Z})$. We set ${\cal M}(\overline{X},A) ={\rm Hom}_c(\pi_1(\overline{X}),A) / A$ be the set of conjugate classes of all continuous homomorphisms $\pi_1(\overline{X}) \rightarrow A$. Recall by Proposition 2.2.5 that we have the fundamental class isomorphism ${{\rm H}}^3({\overline{X}},  \mathbb{Z}/ n \mathbb{Z}) \cong \mathbb{Z}/ n \mathbb{Z}$ that depends on $\zeta_n$.
\begin{dfnS}
For $\rho \in {\cal M}(\overline{X},  A)$, the {\it arithmetic Chern-Simons invariant} $CS_c(\rho)$ of $\rho$ associated to $c$ is defined by the image of $c$ under the composition of the maps
\begin{equation}
{{\rm H}}^3( A,  \mathbb{Z}/n \mathbb{Z}) \xrightarrow{\rho^{\ast}} {{\rm H}}^3( \pi_1(\overline{X}),  \mathbb{Z}/n \mathbb{Z}) \xrightarrow{j_3} {{\rm H}}^3( \overline{X},  \mathbb{Z}/n \mathbb{Z}) \cong \mathbb{Z}/ n \mathbb{Z}, \nonumber
\end{equation}
where $j_3$ is the edge homomorphisms in the modified Hochschild-Serre spectral sequence ${{\rm H}}^p( \pi_1(\overline{X}),  {{\rm H}}^q( \tilde{\overline{X}},  \mathbb{Z}/n \mathbb{Z} )) \Rightarrow {{\rm H}}^{p+q} ( \overline{X},  \mathbb{Z}/n \mathbb{Z} )$ (see the section 2.2 for $\tilde{\overline{X}}$). We can easily see that $CS_c(\rho)$ is independent of the choice of $\rho$ in its conjugate class. The map $$CS_c : {\cal M}(\overline{X},A) \rightarrow \mathbb{Z} / n\mathbb{Z}$$ is called the {\it arithmetc Chern-Simons functional} associated to $c$.
The {\it arithmetic Dijkgraaf-Witten invariant} of $\overline{X}$  associated to $c$ is then defined by
\begin{equation}
\displaystyle{Z_c (\overline{X}) = \sum_{\rho \in  {\cal M}(\overline{X},A)} {\mathrm exp} \left( \frac{2 \pi i}{n} CS_c(\rho) \right)   }. \nonumber
\end{equation}
When $A=\mathbb{Z}/ m \mathbb{Z}$, we call $CS_c(\rho)$ and $Z_c(\overline{X})$ the {\it mod $m$ arithmetic Chern-Simons invariant} and the {\it mod $n$ arithmetic Dijkgraaf-Witten invariant}, respectively.

\end{dfnS}

\begin{rmkS}
(1) If $K$ is totally imaginary,  we have $\pi_{1} (\overline{X}) = \pi_1 (X)$ and ${\rm H}^{i}(\overline{X},  \mathbb{Z}/ n \mathbb{Z}) = {\rm H}^{i}(X,  \mathbb{Z}/ n \mathbb{Z})$ because we don't need considering the ramification at infinite primes. Therefore Definition 3.1 is indeed an extension of Kim's definition([K]).\\
(2) When $A$ is abelian, by Proposition 2.1.6, we have 
$${\cal M}(\overline{X},A)={\rm Hom}_{c}(\pi_1 (\overline{X}),A) ={\rm Hom}({\rm Cl}_K,A).$$
\end{rmkS}

\section{Mod 2 arithmetic Dijkgraaf-Witten invariants for the real quadratic number fields $\mathbb{Q}(\sqrt{p_1 \cdots p_r})$, $p_i\equiv 1$ mod $4$}
\ \ \ In the section 4.1, we extend a formula obtained in [AC] and [BCGKPT], which relates the mod n arithmetic Chern-Simons functional to the Artin symbol, for any number field. Using it and Gauss genus theory, in the section 4.2, we compute explicitly the mod 2 arithmetic Dijkgraaf-Witten invariant for the quadratic fields $\mathbb{Q}(\sqrt{p_1 \cdots p_r})$, $p_i\equiv 1$ mod $4$.
\subsection{A formlula relating with the Artin symbol}
\ \ \ Firstly, we describe the setting in this subsection.  Keeping the same notations as in the section 3, we set $A= \mathbb{Z} / n \mathbb{Z}$ and $c={\mathrm id} \cup \beta({\mathrm id}) \in {\rm H}^3( A,  \mathbb{Z}/n \mathbb{Z})$,  where $id\in {\rm H}^1( A,  \mathbb{Z}/n \mathbb{Z})$ is the identity map and $$\beta : {\rm H}^1( A,  \mathbb{Z}/n \mathbb{Z}) \rightarrow {\rm H}^2( A,  \mathbb{Z}/n \mathbb{Z})$$ is the Bockstein map (connecting homomorphism) induced by the short exact sequence $$0 \rightarrow \mathbb{Z} / n\mathbb{Z} \overset{\times n}{\rightarrow} \mathbb{Z} / n^2 \mathbb{Z} \rightarrow \mathbb{Z} / n\mathbb{Z} \rightarrow 0.$$ Let $j_i :{{\rm H}}^i( \pi_1(\overline{X}),  \mathbb{Z}/n \mathbb{Z}) \rightarrow {{\rm H}}^i( \overline{X},  \mathbb{Z}/n \mathbb{Z}) \ (i=0,1,2,3,\cdots) $ be the edge homomorphisms in the modified Hochschild-Serre spectral sequence (Corollary 2.2.7).  We will abbreviate $j_i \circ \rho^{*}$ to $\rho^{*}_{X}$ for $ \rho \in {\cal M}(\overline{X},A)={\rm Hom}_{c}(\pi_1 (\overline{X}),A)$.  Then  we have
\begin{equation}
CS_c(\rho)=\rho^{*}_{X} ({\mathrm id}) \cup \tilde{\beta} (\rho^{*}_{X} ({\mathrm id})),  \nonumber
\end{equation}
where $\cup : {\rm H}^1( \overline{X},  \mathbb{Z}/n \mathbb{Z}) \times {\rm H}^2( \overline{X},  \mathbb{Z}/n \mathbb{Z}) \rightarrow {\rm H}^3( \overline{X},  \mathbb{Z}/n \mathbb{Z})$ is the cup product and $\tilde{\beta} :  {\rm H}^1( \overline{X},  \mathbb{Z}/n \mathbb{Z}) \rightarrow {\rm H}^2( \overline{X},  \mathbb{Z}/n \mathbb{Z})$ is the Bockstein map.
\begin{rmk}
As for cup products in the category of sheaves on any site, we refer to  [Sw; Corollary 3.7].
\end{rmk}
Before stating the formula relating with the Artin symbol, we recall some calculations on the cohomology of groups.
\begin{lem}
We have $${\rm H}^i( \mathbb{Z}/n \mathbb{Z},  \mathbb{Z}/n \mathbb{Z}) = \mathbb{Z}/n \mathbb{Z} \  (i\geq0)$$ and ${\rm H}^3( \mathbb{Z}/n \mathbb{Z},  \mathbb{Z}/n \mathbb{Z})$ is generated by $c={\mathrm id} \cup \beta({\mathrm id})$. The cohomology class $c$ is represented by the cochain $\alpha : (\mathbb{Z}/ n \mathbb{Z})^3 \rightarrow \mathbb{Z}/n \mathbb{Z}$ defined by 
$$\alpha (g_1,g_2,g_3)=\frac{1}{n} \overline{g_1} (\overline{g_2}+\overline{g_3}-(\overline{g_2+g_3})) \ {\rm mod}\ n,$$
where $\overline{g_i} \in \{ 0,1,\dots,n-1 \}$ such that $\overline{g_i} \  {\rm mod}\ n=g_i.$
\end{lem}
\begin{proof}
Although this may be well known, we give a proof for the sake of readers. To calculate ${\rm H}^i( \mathbb{Z}/n \mathbb{Z},  \mathbb{Z}/n \mathbb{Z})$ and ${\rm H}^i( \mathbb{Z}/n \mathbb{Z}, \mathbb{Z})$, we take a projective resolution of $\mathbb{Z}$ as follows,
$$\cdots \overset{\times n}{\rightarrow} \mathbb{Z}[\mathbb{Z}/n \mathbb{Z}] \overset{\times 0}{\rightarrow} \mathbb{Z}[\mathbb{Z}/n \mathbb{Z}] \overset{\times n}{\rightarrow} \mathbb{Z}[\mathbb{Z}/n \mathbb{Z}] \overset{\times 0}{\rightarrow} \mathbb{Z}[\mathbb{Z}/n \mathbb{Z}] \overset{\epsilon}{\rightarrow} \mathbb{Z}, $$
where $\epsilon$ is defined by $\sum_{g} a_g g \mapsto \sum_{g} a_g$. Therefore, by taking the functor ${\rm Hom}_{\mathbb{Z}[\mathbb{Z}/n \mathbb{Z}]}(-,\mathbb{Z}/n \mathbb{Z})$, we have \\\\
(4.1.2.1)\ \ \ \ \ \ \ \ \ \ \ \ \ \  ${\rm H}^i( \mathbb{Z}/n \mathbb{Z},  \mathbb{Z}/n \mathbb{Z}) = \mathbb{Z}/n \mathbb{Z}\ (i\geq0).$ \\\\
By the same manner, we obtain\\\\
(4.1.2.2)\ \ \ \ \ \ \ \ \ \ \ \ \ \  ${\rm H}^i( \mathbb{Z}/n \mathbb{Z},  \mathbb{Z}) =
\begin{cases} 
\mathbb{Z} & (i=0)\\
\mathbb{Z}/n \mathbb{Z} & (i\geq1,\ i {\rm :even})\\
0 & (i\geq1,\ i {\rm :odd}).\\
\end{cases}$ \\ \par Next, we consider the following commutative diagram,
\begin{center}
\[
  \begin{CD}
     0 @>>>  \mathbb{Z}  @>{\times n}>>  \mathbb{Z}  @>>>  \mathbb{Z}/ n\mathbb{Z}  @>>>  0 \\
    @.     @VV{p_1}V  @VV{p_2}V  @VV{id}V   @. \\
     0 @>>>  \mathbb{Z}/n \mathbb{Z} @>{\times n}>>  \mathbb{Z}/n^2 \mathbb{Z} @>>>  \mathbb{Z}/n \mathbb{Z} @>>>  0,
  \end{CD}
\]
\end{center}
where $p_1$ and $p_2$ are   natural projections, and $id$ is the identity map. We see that the connecting homomorphism $\beta' : {\rm H}^{1}( \mathbb{Z}/n \mathbb{Z},  \mathbb{Z} /n \mathbb{Z}) \rightarrow  {\rm H}^{2}( \mathbb{Z}/n \mathbb{Z},  \mathbb{Z})$ associated to the 1st row is an isomorphism, because of  (4.1.2.1) and (4.1.2.2). We also see that the homomorphism ${p_1}_{*} : {\rm H}^{2}( \mathbb{Z}/n \mathbb{Z},  \mathbb{Z} ) \rightarrow {\rm H}^{2}( \mathbb{Z}/n \mathbb{Z},  \mathbb{Z}/n \mathbb{Z})$ induced by $p_1$ is an isomorphism, because of (4.1.2.1) and (4.1.2.2). Therefore, the connecting homomorphism $\beta : {\rm H}^{1}( \mathbb{Z}/n \mathbb{Z},  \mathbb{Z} /n \mathbb{Z}) \rightarrow  {\rm H}^{2}( \mathbb{Z}/n \mathbb{Z},  \mathbb{Z}/n \mathbb{Z})$ associated to the 2nd row, that is the composition of  ${p_1}_{*}$ and $\beta'$, is an isomorphism. So  ${\rm H}^{2}( \mathbb{Z}/n \mathbb{Z},  \mathbb{Z}/n \mathbb{Z})$ is generated by $\beta(id)$. On the other hand, by the construction of the cup product based on the projective resolution, we can see that the paring  $\cup : {\rm H}^{1}( \mathbb{Z}/n \mathbb{Z},  \mathbb{Z}/n \mathbb{Z}) \times {\rm H}^{2}( \mathbb{Z}/n \mathbb{Z},  \mathbb{Z}/n \mathbb{Z}) \rightarrow {\rm H}^{3}( \mathbb{Z}/n \mathbb{Z},  \mathbb{Z}/n \mathbb{Z} ) $ is not a zero map. Therefore ${\rm H}^3( \mathbb{Z}/n \mathbb{Z},  \mathbb{Z}/n \mathbb{Z})$ is generated by $c={\mathrm id} \cup \beta({\mathrm id})$. The last assertion follows from the definition of the Bockstein map and the cup product of the group cohomology.
\end{proof}
Now we move to state the main assertion in this subsection.
\begin{prop}
Let $\rho : \pi_1(\overline{X}) \rightarrow \mathbb{Z}/ n \mathbb{Z}$ be a continuous and surjective homomorphism. We set  $A=\mathbb{Z} / n \mathbb{Z}$ and $c={\mathrm id} \cup \beta({\mathrm id}) \in {\rm H}^3( A,  \mathbb{Z}/n \mathbb{Z})$. Let $L = K(v^{\frac{1}{n}})$ be the  Kummer extension unramified at all finite and infinite primes  corresponding to the kernel of $\rho$ such that there exists $ \mathfrak{a} \in I_K$ satisfying ${\mathfrak{a}}^n =(v)^{-1} $.  We take a generator $ \sigma \in {\rm Gal}(L/K)$  by $\sigma(v^{\frac{1}{n}})/ v^{\frac{1}{n}} = \zeta_n$, where $\zeta_n$ is the primitive $n$-th root of unity chosen before Definition 3.1, and let $\chi : {\rm Gal}(L/K) \stackrel{\sim}{\rightarrow} \mathbb{Z}/ n\mathbb{Z}$ be an isomorphism defined by  $\chi(\sigma) = 1$. Then we have
$$ CS_c(\rho)=\chi \left( \left( \frac{L/K}{\mathfrak{a}} \right) \right).$$
\end{prop}
\begin{proof}
When $K$ is totally imaginary, the  statement is known by [BCGKPT; Theorem 1.3]. So we consider the case $K$ has real places and $n=2$. By Remark 2.2.6, we can identify $\rho^{*}_{X} ({\mathrm id}) \in {\rm H}^1( \overline{X},  \mathbb{Z}/2 \mathbb{Z})$ with Artin map  $\left( \frac{L/K}{ } \right)$. To calculate the cup product $ \rho^{*}_{X} ({\mathrm id}) \cup \tilde{\beta} (\rho^{*}_{X} ({\mathrm id}))$,  we refer to [AC]. We regard $ \tilde{\beta} (\rho^{*}_{X} ({\mathrm id}))$ as an element in ${\rm Ext}^{1}_{\overline{X}} (\mathbb{Z} / 2 \mathbb{Z},  \phi_{*} G_{m,X})^{\sim} = (Z_1/B_1)^{\sim}$ through Artin-Verdier duality. Then by [AC; Corollary 3.10] we have 
\begin{equation}
\rho^{*}_{X} ({\mathrm id}) \cup \tilde {\beta} (\rho^{*}_{X} ({\mathrm id}))= \tilde{\beta} (\rho^{*}_{X} ({\mathrm id})) ([(v,\mathfrak{a})]) = \rho^{*}_{X} ({\mathrm id})(\tilde{\beta}'([(v,  \mathfrak{a})])), \nonumber
\end{equation}
where $\tilde{\beta}' : {\rm Ext}^{1}_{\overline{X}} (\mathbb{Z} / 2 \mathbb{Z},  \phi_{*} G_{m,X}) \rightarrow {\rm Ext}^{2}_{\overline{X}} (\mathbb{Z} / 2 \mathbb{Z},  \phi_{*} G_{m,X})$ is the connecting homomorphism induced by the short exact sequence $$0 \rightarrow \mathbb{Z} / 2\mathbb{Z} \overset{\times 2}{\rightarrow} \mathbb{Z} / 2^2 \mathbb{Z} \rightarrow \mathbb{Z} / 2\mathbb{Z} \rightarrow 0.$$
By replacing $X$ with $\overline{X}$ in the proof of [AC; Lemma 4.1],  one can see $\tilde{\beta}'([(v,  \mathfrak{a})]) = [\mathfrak{a}]$. Therefore we see that $ CS_c(\rho) = 0 $ holds iff $\left( \frac{L/K}{\mathfrak{a}} \right) \in {\rm Gal}(L/K)$ is trivial.  Combining this fact and [BCGKPT; Theorem 1.3], we have the required statement.
\end{proof}

\subsection{Explicit formulas of the mod 2 arithmetic Dijkgraaf-Witten invariants  for real quadratic number fields $\mathbb{Q} (\sqrt{p_1  p_2  \cdots  p_r})$,  $ p_i \equiv 1$ mod $4$ }
\ \ \ In the following,  we consider the case $K=\mathbb{Q} (\sqrt{p_1  p_2  \cdots  p_r})  $, where $p_i$ is a prime number such that $ p_i \equiv 1$ mod $4$. Let $n=2$, $A= \mathbb{Z} / 2 \mathbb{Z}$ and $c={\mathrm id} \cup \beta({\mathrm id}) \in {\rm H}^3( A,  \mathbb{Z}/2 \mathbb{Z})$. Assume that the norm of the fundamental unit in ${\cal O}^{\times}_K$ is $-1$. Then the narrow ideal class group ${\rm Cl}^{+}_K$ is same as ${\rm Cl}_K$. Note that the discriminant of $K$ is $p_1 p_2 \cdots p_r$ because of the assumption, $ p_i \equiv 1$ mod $4$. We define the abelian multiplicative 2-group $T_{\times}$ by 
$$T_{\times}= \{ ( x_1,  x_2,  \cdots,  x_r) \in {\{ \pm 1 \}}^r \;|\; \displaystyle{\prod^{r}_{i=1} x_i } = 1 \},$$
and let $e^{\times}_{ij} \overset{def}{=}(1,  \cdots, 1,\overset{i \mbox{-th}}{-1},1,  \cdots, 1,\overset{j \mbox{-th}}{-1},1,  \cdots, 1) \in T_{\times} \ {\rm for}\ 1 \leqq i < j \leqq r $. We also identify $T_{\times}$ with the abelian additive 2-group $T_{+}$ defined by 
$$T_{+}=\{ ( x_1,  x_2,  \cdots,  x_r) \in {(\mathbb{Z}/2 \mathbb{Z})}^r \;|\; \displaystyle{\sum^{r}_{i=1} x_i } = 0 \}.$$and let $e^{+}_{ij} \overset{def}{=}(0,  \cdots, 0,\overset{i \mbox{-th}}{1},0,  \cdots, 0,\overset{j \mbox{-th}}{1},0,  \cdots, 0) \in T_{+} \ {\rm for}\ 1 \leqq i < j \leqq r $.\par
By Gauss genus theory ([O; $\S 4.7$]), there is an isomorphism
\begin{equation}
{\rm Cl}^{+}_K / 2 {\rm Cl}^{+}_K \stackrel{\sim}{\longrightarrow} T_{\times},  
\end{equation}
given by
\begin{equation}
[\mathfrak{a}] \mapsto \left( \left( \frac{{\rm N}\mathfrak{a}}{p_1} \right),  \left( \frac{{\rm N}\mathfrak{a}}{p_2} \right),  \cdots,  \left( \frac{{\rm N}\mathfrak{a}}{p_r} \right)  \right), \nonumber
\end{equation}
where $\left( \frac{ }{p_i} \right)$ denotes the Legendre symbol. Then we obtain the identifications
\begin{eqnarray}
{\rm Hom}_{c}(\pi_1(\overline{X}),\mathbb{Z}/2\mathbb{Z})&=&{\rm Hom} ({\rm Cl}^{+}_K / 2 {\rm Cl}^{+}_K, \mathbb{Z} / 2\mathbb{Z}) \nonumber \\
                                                                         &=&{\rm Hom}( T_{\times} , {\{ \pm 1 \}} )={\rm Hom}( T_{+} , \mathbb{Z} / 2\mathbb{Z} ), \nonumber 
\end{eqnarray}
by Proposition 2.1.6.\par
Now we move to prove the following formula.

\setcounter{dfn}{1}
\begin{thm}Notations being as above, for $\rho \in {\rm Hom} ( T_{\times} , {\{ \pm 1 \}} )$, we have
\begin{equation}
(-1)^{CS_c(\rho)}=\displaystyle{\underset{\rho(e^{\times}_{ij})=-1}{ {\prod_{ i<j} } } } \left( \frac{p_j}{p_i} \right). \nonumber
\end{equation}
\end{thm}
\begin{proof}
We take a basis of $T_{\times}$ over ${\{ \pm 1 \}}$, $b_1,b_2,\cdots,b_{r-1}$,  which is defined by $$b_1=(-1,1,1,\cdots,-1),  b_2 =(1,-1,1,1,\cdots,-1),  \cdots, b_{r-1} = (1,1,\cdots,1,-1,-1).$$ We define $1 \leqq j_1<j_2\cdots<j_m \leqq r-1$ by
\begin{equation}
\rho(b_i)=
\begin{cases} 
-1 & (i=j_1,j_2,\cdots,j_m)\\
1 & ({\rm otherwise}).\\
\end{cases} \nonumber
\end{equation}
 Note that $\rho(e^{\times}_{ij})=-1$ holds iff  only one of $i$ and $j$ is in $\{j_1,j_2,\cdots,j_m \}.$ Let $L$ be the abelian unramified extension of $K$ corresponding to $2{\rm Cl}_K$ by class field theory (Proposition 2.1.6), precisely, $$L=\mathbb{Q} (\sqrt{p_1},   \sqrt{p_2},  \cdots \sqrt{ p_r} ).$$ Let $L_{\rho}$ be the unramified Kummer extension of $K$ corresponding to the kernel of $\rho \in {\rm Hom}_{c}(\pi_1(\overline{X}),\mathbb{Z}/2\mathbb{Z}).$ So we have 
$$
L_{\rho}=K(\sqrt{v}), \;\; \mathfrak{a}_v^2 = (v)^{-1}, \leqno{(4.2.2.1)} 
$$
where $\mathfrak{a}_v$ is a fractional ideal of $K$.
To use Proposition 4.1.3, we firstly determine this $v$. For $(a_1,a_2,\cdots,a_r) \in T_{\times}$,  we denote by $\mathfrak{a} (a_1,a_2,\cdots,a_r) \in {\rm Cl}^{+}_K / 2 {\rm Cl}^{+}_K$ the corresponding class via (4.2.1). Then for any $(a_1,a_2,\cdots, a_r) \in T_{\times}$,  $\left( \frac{L/K}{\mathfrak{a} (a_1,a_2,\cdots,a_r) } \right) \in {\rm Gal}(L/K)$ is characterized by
\begin{equation}
\left( \frac{L/K}{\mathfrak{a} (a_1,a_2,\cdots,a_r) } \right) (\sqrt {p_i})=a_i \sqrt {p_i}\;\;  (i=1,2,\cdots,r). \nonumber
\end{equation}
 On the other hand, by Remark 2.2.6, $v$ is characterized by
\begin{equation}
\left( \frac{L_{\rho}/K}{\mathfrak{a} (a_1,a_2,\cdots,a_r) } \right) (\sqrt {v}) / \sqrt {v}=\rho(a_1,a_2,\cdots,a_r), \nonumber
\end{equation}
for $(a_1,a_2,\cdots,a_r) \in T_{\times}$.
Since  $\left( \frac{L_{\rho}/K}{} \right)$ is the restriction of $\left( \frac{L/K}{} \right)$ to $L_\rho$,  we can take $v$ as $v= p_{j_1}p_{j_2} \cdots p_{j_m} / p_1 p_2 \cdots p_r$. 
 Since the minimal polynomial of $(1+\sqrt{p_1  p_2  \cdots  p_r})/2$ over $\mathbb{Q}$ is congruent to $(2X-1)^2$ in mod $p_i$, we have
$$(p_i)= {\mathfrak{p}_i}^2,$$
where $\mathfrak{p}_i = (p_i,  \sqrt{p_1 p_2 \cdots p_r}) $ is the prime ideal of ${\cal O}_K.$ Then we can take $\mathfrak{a}_v$ in (4.2.2.1) as 
$$\mathfrak{a}_{v} =\mathfrak{p}_1 \mathfrak{p}_2 \cdots \mathfrak{p}_r / \mathfrak{p}_{j_1} \mathfrak{p}_{j_2} \cdots \mathfrak{p}_{j_m}.$$
We take a generator $ \sigma \in {\rm Gal}(L_{\rho}/K)$  by $\sigma(\sqrt{v})/\sqrt{v}=-1$, and  let $\chi : {\rm Gal}(L_{\rho}/K) \stackrel{\sim}{\rightarrow} \mathbb{Z}/ 2\mathbb{Z}$ be an isomorphism defined by  $\chi(\sigma) = 1$. Then, by Proposition 4.1.3,  we have\\
\begin{eqnarray}
CS_c(\rho) &=&\chi \left( \left( \frac{L_{\rho}/K}{\mathfrak{a}_{v} } \right) \right) \nonumber \\
               &=&\chi \left( \left( \frac{L_{\rho}/K}{\mathfrak{a} \left(  \left( \frac{{\rm N}\mathfrak{a}_v}{p_1} \right),  \left( \frac{{\rm N}\mathfrak{a}_v}{p_2} \right),  \cdots,  \left( \frac{{\rm N}\mathfrak{a}_v}{p_r} \right)   \right)} \right) \right). \nonumber 
\end{eqnarray}
Therefore, we have\\
\begin{eqnarray}
(-1)^{CS_c(\rho)} &=& \left( \frac{L_{\rho}/K}{\mathfrak{a} \left(  \left( \frac{{\rm N}\mathfrak{a}_v}{p_1} \right),  \left( \frac{{\rm N}\mathfrak{a}_v}{p_2} \right),  \cdots,  \left( \frac{{\rm N}\mathfrak{a}_v}{p_r} \right)   \right)} \right) \left( \sqrt{v} \right) \bigg/ \sqrt{v} \nonumber \\
               &=&\displaystyle{\prod^{m}_{l=1} \left( \frac{{\rm N}\mathfrak{a}_v}{p_{j_l}} \right)} \nonumber \\
               &=&\displaystyle{\prod^{m}_{l=1} \underset{i \notin \{ j_1,j_2,\cdots,j_m \} }{ \prod_{1\leq i \leq r} }\left( \frac{p_i}{p_{j_l}} \right)} \nonumber \\
               &=&\displaystyle{\underset{\rho(e^{\times}_{ij})=-1}{ {\prod_{ i<j} } } } \left( \frac{p_j}{p_i} \right).  \nonumber 
\end{eqnarray}
\end{proof}

If we rewrite Proposition4.2.2 for the form $\rho : T_{+} \rightarrow \mathbb{Z}/2\mathbb{Z}$,  we have the following.
\begin{coro}Notations being as above, for $\rho \in {\rm Hom} (T_{+},\mathbb{Z}/2 \mathbb{Z})$, we have
\begin{equation}
CS_c(\rho)=\sum_{i < j} \rho (e^{+}_{ij}) {\rm lk}_2( p_i,  p_j ), \nonumber
\end{equation}
where ${\rm lk}_2( p_i,  p_j )$ is the modulo $2$ linking number of $p_i$ and $p_j$ defined by $(-1)^{{\rm lk}_2(p_i,p_j)} = \left( \frac{p_i}{p_j} \right)$.
\end{coro}
By Definition 3.1, the mod 2 arithmetic Dijkgraaf-Witten invariant of $\overline{X}$ associated to $c$ is given as follows.
\begin{coro}
Notations being as above,  we have
\begin{equation}
Z_c(\overline{X})=\sum_{\rho \in {\rm Hom}( T_{+} , \mathbb{Z}/2 \mathbb{Z})} \left( \prod_{i < j}  {\left( \frac{p_i}{p_j} \right)}^ {\rho (e^{+}_{ij})} \right). \nonumber 
\end{equation}
\end{coro}
\begin{exm}
Here are some numerical examples of $CS_c(\rho)$ and $Z_c(\overline{X})$ for the case $r=3$. 
We define $\rho_0$, $\rho_1$, $\rho_2$ and $\rho_3$ in ${\rm Hom}(T_{+},\mathbb{Z}/ 2\mathbb{Z})$ by
$$\rho_0(1,1,0)=0,\  \rho_0(0,1,1)=0,\  \rho_0(1,0,1)=0,$$ 
$$\rho_1(1,1,0)=1,\  \rho_1(0,1,1)=0,\  \rho_1(1,0,1)=1,$$ 
$$\rho_2(1,1,0)=0,\  \rho_2(0,1,1)=1,\  \rho_2(1,0,1)=1,$$ 
$$\rho_3(1,1,0)=1,\  \rho_3(0,1,1)=1,\  \rho_3(1,0,1)=0,$$
so that ${\rm Hom}(T_{+},\mathbb{Z}/ 2\mathbb{Z})=\{ \rho_0, \rho_1, \rho_2, \rho_3 \}.$\\
(1) $K=\mathbb{Q}(\sqrt{5 \cdot 29 \cdot 37}):$
$${\rm lk}_2(5,29)=0, \ {\rm lk}_2(29,37)=1, \ {\rm lk}_2(37,5)=1,$$ 
$$CS_c(\rho_0)=0,\  CS_c(\rho_1)=1,\  CS_c(\rho_2)=0,\  CS_c(\rho_3)=1,$$ 
$$Z_c(\overline{X})=0.$$ 
(2) $K=\mathbb{Q}(\sqrt{5 \cdot 13 \cdot 73}):$
$${\rm lk}_2(5,13)={\rm lk}_2(13,73)={\rm lk}_2(73,5)=1,$$
$$CS_c(\rho_0)=CS_c(\rho_1)=CS_c(\rho_2)=CS_c(\rho_3)=0,$$
$$Z_c(\overline{X})=4.$$
\end{exm}

\section{Appendix on mod $2$ Dijkgraaf-Witten invariants for double branched covers of the 3-sphere}
\ \ \ In this appendix, we present  topological analogues of Theorem 4.2.2,  Corollary 4.2.3 and Corollary 4.2.4 in the content of Dijkgraaf-Witten theory for 3-manifolds. For this, we firstly recall the following ${\rm M^2KR}$-dictionary, due to Mazur, Morishita , Kapranov and Reznikov, concerning  the analogies  between 3-dimensional topology and number theory (cf. [Mo]).

\begin{table}[htb]
\begin{center}
 \begin{tabular}{|c||c|} \hline
  3-dimensional topology & number theory \\ \hline \hline
  connected, oriented and closed & compactified  spectrum of a numbr ring  \\
  3-manifold $M$ &   $\overline{X}= \overline{{\rm Spec}\ {\cal O}_k} $ \\ \hline\
  knot & maximal ideal  \\  
  $K\ : \ S^1 \rightarrow M $ & $\mathfrak{p}={\rm Spec}\ ({\cal O}_k/\mathfrak{p}) \rightarrow {\rm Spec}\ ({\cal O}_k)$ \\ \hline
  link & finite set of maximal ideals \\
  $L=K_1\cup K_2 \cup \cdots \cup K_r$ & $S=\{\mathfrak{p_1},\mathfrak{p_2},\cdots,\mathfrak{p_r}\}$ \\ \hline 
  fundamental group & modified \'etale fundamental group  \\ 
  $\pi_1(M)$ &  $\pi_1(\overline{X})$ \\ \hline
  1-cycle group $Z_1(M)$ & ideal group $I_K$ \\
  1-boundary group $B_1(M)$ & principal ideal group $P_K$ \\
  $\partial : C_2(M) \rightarrow Z_1(M) ; S \mapsto \partial S$ & $\partial : K^{\times} \rightarrow I_K ; a \mapsto (a) $ \\ \hline
  1st integral homology group & ideal class group \\
  ${\rm H}_1(M)=Z_1(M)/B_1(M)$ &  ${\rm Cl}_K=I_K/P_K$ \\ \hline
  Hurewicz isomorphism & Artin reciprocity \\
  ${\pi_1(M)}^{ab} \cong {\rm Gal}(M^{ab} / M)  \cong {\rm H}_1(M)$ & $\pi_1(\overline{X})^{ab} \cong {\rm Gal} ({\tilde{K}}^{ab}/K) \cong {\rm Cl}_K $ \\ \hline
  Poincar\'e duality & Artin-Verdier duality \\
  ${\rm H}^i(M, \mathbb{Z} / n \mathbb{Z}) \cong {\rm H}_{3-i}(M , \mathbb{Z} / n \mathbb{Z})$ & ${\rm H}^i(\overline{X}, \mathbb{Z} / n \mathbb{Z}) \cong { {\rm Ext}^{3-i}_{\overline{X}}( \mathbb{Z} / n \mathbb{Z}, \phi_{*} G_{m,X}) }^{\sim}$ \\ \hline
  
 \end{tabular}
\end{center}
\end{table}

In Section 5.1, we introduce the Dijkgraaf-Witten invariants for 3-manifolds. Following  the ${\rm M^2KR}$-dictionary, in Section 5.2, we show a topological analogue of Theorem 4.1.3 and in Section 5.3, we prove  topological analogues of  Corollary 4.2.3 and Corollary 4.2.4.

\subsection{Dijkgraaf-Witten invariants for 3-manifolds}
\ \ \ In this subsection, we introduce the Dijkgraaf-Witten invariants in a manner, which is slightly different from the original one ([DW]), in order to  clarify the analogy between the Dijkgraaf-Witten invariant for a 3-manifold and the arithmetic Dijkgraaf-Witten invariant for a number ring. As a preparation of defining the invariant, we firstly show the following proposition that is the topological analogue of Corollary 2.2.8.
\begin{prop}
Let $M$ be a connected, compact 3-manifold. Then, for $n\geq2$, there is an cohomological spectral sequence
\begin{equation}
{\rm H}^p( \pi_1(M),  {\rm H}^q( \tilde{M},  \mathbb{Z}/n \mathbb{Z} )) \Rightarrow {\rm H}^{p+q} (M,  \mathbb{Z}/n \mathbb{Z} ), \nonumber
\end{equation}
where $\tilde{M}$ is the universal covering of $M$.
\end{prop}
\begin{proof}
Since $M$ is compact, the singular cohomology ${\rm H}^{i}(M,\mathbb{Z}/ n \mathbb{Z})$ can be identified with the cohomology of the constant sheaf $\mathbb{Z}/ n \mathbb{Z}$ on $M$. So we show the required statement for the cohomology of the constant sheaf.  We denote by ${\rm Gal} (\tilde{M} / M)$-mod  the category of ${\rm Gal} (\tilde{M} / M) \mbox{-}$modules. We consider the following functors
$$ \begin{array}{c}
F_1 : {\rm Sh} (M) \rightarrow {\rm Gal} (\tilde{M} / M) \mbox{-mod},  S \mapsto S(\tilde{M})\\
F_2 : {\rm Gal} (\tilde{M} / M)\mbox{-mod} \rightarrow {\rm Ab},  R \mapsto R^{{\rm Gal}(\tilde{M}/M) }
\end{array} $$
where the action of $G={\rm Gal} (\tilde{M}/M)$ on $S(\tilde{M})$ is defined by $\sigma . x =S(\sigma) (x) $ for $x \in S(\tilde{M})$ and $\sigma \in G$.  In the same way as Proposition 2.2.7,  we can  check $(F_2 \circ F_1) (S) = S(\tilde{M})^{G} = S(M)$ and that $F_1$ takes any injective object $I$ to a $F_2$-acyclic object.  So we have the required spectral sequence by the Grothendieck spectral sequence and $\pi_1(M) \cong {\rm Gal}(\tilde{M}/M)$.
\end{proof}
Now we define the Dijkgraaf-Witten invaritnat for a 3-manifold as follows.
\begin{dfn}
Let $M$ be a connected, oriented and closed 3-manifold, let $n\geq2$, let $A$ be a finite group and let $c \in {{\rm H}}^3 (A,  \mathbb{Z}/ n \mathbb{Z})$. We set ${\cal M}(M,A) ={\rm Hom}(\pi_1(M),A)/A$. Since $M$ is compact and oriented, there is an isomorphism ${\rm H}_3(M,\mathbb{Z}/ n \mathbb{Z}) \cong \mathbb{Z}/ n \mathbb{Z}$ and we denote by $[M] \in {\rm H}_3(M,\mathbb{Z}/ n \mathbb{Z})$ the fundamental homology class of $M$. For $\rho \in {\cal M}(M, A)$, the {\it  Chern-Simons invariant}  $CS_c(\rho)$ of $\rho$ associated to $c$ is defined by the image of $c$ under the composition of the maps
\begin{equation}
{{\rm H}}^3( A,  \mathbb{Z}/n \mathbb{Z}) \xrightarrow{\rho^{\ast}} {{\rm H}}^3( \pi_1(M),  \mathbb{Z}/n \mathbb{Z}) \xrightarrow{j_3} {{\rm H}}^3( M,  \mathbb{Z}/n \mathbb{Z}) \xrightarrow{<\ \ ,[M]>} \mathbb{Z}/ n \mathbb{Z}, \nonumber
\end{equation}
where $j_3$ is the edge homomorphisms in the spectral sequence of Proposition 5.1.1 ${\rm H}^p( \pi_1(M),  {\rm H}^q( \tilde{M},  \mathbb{Z}/n \mathbb{Z} )) \Rightarrow {\rm H}^{p+q} (M,  \mathbb{Z}/n \mathbb{Z} )$. The {\it Dijkgraaf-Witten invariant} of $M$  associated to $c$ is then defined by
\begin{equation}
\displaystyle{Z_c (M) = \sum_{\rho \in  {\cal M}(M,A)} {\mathrm exp} \left( \frac{2 \pi i}{n} CS_c(\rho) \right)   }. \nonumber
\end{equation}
When $A=\mathbb{Z}/ m \mathbb{Z}$, we call $CS_c(\rho)$ and $Z_c(M)$ the {\it mod $m$ Chern-Simons invariant} and the {\it mod $n$ Dijkgraaf-Witten invariant}, respectively.
\end{dfn}
\begin{rmk}
The Dijkgraaf-Witten invariant was originally defined as follows ([DW]). Let $M$ be a connected, oriented and closed 3-manifold. Let $U(1)=\{ z\in \mathbb{C} \ |\ |z|=1   \}$. We denote by $[M] \in {\rm H}_3(M,\mathbb{Z})$ the   fundamental homology class of $M$.  Let $A$ be a finite group.  Let $c \in {{\rm H}}^3 (A,  U(1))$. Then the Dijkgraaf-Witten invariant, denoted by $DW_c(M)$, is defined by $$ DW_c(M) ={\displaystyle \underset{\rho \in {\rm Hom}(\pi_1(M),A)}{\sum} <{f_\rho}^* c , [M]>}, $$  where $f_\rho : M \rightarrow {\rm B}A$ is a classifying  map with respect to $\rho$ and $<\ ,\ > \ : {{\rm H}}^3 (M,  U(1)) \times {{\rm H}}_3 (M,  \mathbb{Z}) \rightarrow U(1)$ is the natural pairing. The relation between this definition and Definition 5.1.2 is given as follows. Let $A=\mathbb{Z}/n\mathbb{Z}$. By the isomorphism $ {{\rm H}}^3 (A,  U(1)) \cong \mu_n \subset U(1)$,  we can regard $c$ as an \ $n$-th root of unity $\zeta_{n,c}$ in $U(1)$. Then, one can check that, for any $\rho \in {\rm Hom}(\pi_1(M),A)$, we have $$\zeta_{n,c}^{CS_{id \cup \beta (id)}(\rho)} \cong <{f_\rho}^* c , [M]>,$$ where $id$ and $\beta$ are defined as in Section 4.1.  In paticular, when $\zeta_{n,c}={\rm exp}(\frac{2 \pi i}{n})$,  we have $$DW_c(M)=Z_{{\mathrm id} \cup \beta({\mathrm id})}(M).$$
\end{rmk}

\subsection{The formlula relating with Hurewicz isomorphism}
\ \ \ In this subsection, we show the topological analogue of Proposition 4.1.3.  We firstly describe the setting in this subsection.  Keeping the same notations as  Section 5.1, we set $A= \mathbb{Z} / n \mathbb{Z}$ and $c={\mathrm id} \cup \beta({\mathrm id}) \in {\rm H}^3( A,  \mathbb{Z}/n \mathbb{Z})$,  where $id\in {\rm H}^1( A,  \mathbb{Z}/n \mathbb{Z})$ is the identity map and $$\beta^i : {\rm H}^i( A,  \mathbb{Z}/n \mathbb{Z}) \rightarrow {\rm H}^{i+1}( A,  \mathbb{Z}/n \mathbb{Z})\ (i=0,1,2,\cdots)$$ is the Bockstein map (connecting homomorphism) induced by the short exact sequence $$0 \rightarrow \mathbb{Z} / n\mathbb{Z} \overset{\times n}{\rightarrow} \mathbb{Z} / n^2 \mathbb{Z} \rightarrow \mathbb{Z} / n\mathbb{Z} \rightarrow 0.$$ We denote the Bockstein maps with respect to the singular homology and cohomology induced by the short exact sequence $$0 \rightarrow \mathbb{Z} / n\mathbb{Z} \overset{\times n}{\rightarrow} \mathbb{Z} / n^2 \mathbb{Z} \rightarrow \mathbb{Z} / n\mathbb{Z} \rightarrow 0,$$  by $$\beta^i : {\rm H}^i( M,  \mathbb{Z}/n \mathbb{Z}) \rightarrow {\rm H}^{i+1}( M,  \mathbb{Z}/n \mathbb{Z}),\ \ \beta_i : {\rm H}_i( M,  \mathbb{Z}/n \mathbb{Z}) \rightarrow {\rm H}_{i-1}( M,  \mathbb{Z}/n \mathbb{Z}),$$ for $i=0,1,2,\cdots$. We also denote the Bockstein map with respect to the singular homology induced by the short exact sequence $$0 \rightarrow \mathbb{Z}\overset{\times n}{\rightarrow} \mathbb{Z} \rightarrow \mathbb{Z} / n\mathbb{Z} \rightarrow 0,$$  by $$\tilde{\beta}_i : {\rm H}_i( M,  \mathbb{Z}/n \mathbb{Z}) \rightarrow {\rm H}_{i-1}( M,  \mathbb{Z}),$$ for $i=0,1,2,\cdots$.  Let $j_i :{{\rm H}}^i( \pi_1(M),  \mathbb{Z}/n \mathbb{Z}) \rightarrow {{\rm H}}^i( M,  \mathbb{Z}/n \mathbb{Z}) \ (i=0,1,2,3,\cdots) $ be the edge homomorphisms in the  spectral sequence of Proposition 5.1.1.  We will abbreviate $j_i \circ \rho^{*}$ to $\rho^{*}_{M}$ for $ \rho \in {\cal M}(M,A)={\rm Hom}(\pi_1 (M),A)/A$. We denote by $\Phi^i : {\rm H}^i(M , \mathbb{Z} / n \mathbb{Z}) \stackrel{\sim}{\rightarrow} {\rm H}_{3-i}(M,\mathbb{Z}/n \mathbb{Z})\ (i=0,1,2,3)$ the isomorphism of the Poincar\'e duality defined by $u \mapsto u \cap [M]$, where 
$$\cap : {\rm H}^i(M,\mathbb{Z}/n \mathbb{Z}) \times {\rm H}_3(M,\mathbb{Z}/n \mathbb{Z}) \rightarrow {\rm H}_{3-i}(M,\mathbb{Z}/n \mathbb{Z})$$ is the cap product. Note that, by the universal coefficient theorems, we have $${\rm H}_1(M,\mathbb{Z}/n\mathbb{Z}) \cong {\rm H}_1(M) \otimes \mathbb{Z}/ n\mathbb{Z} \cong {\rm H}_1(M) / n {\rm H}_1(M).$$ Then, by the Hurewicz isomorphism, we obtain the following identification,
\begin{eqnarray}
{\rm Hom}(\pi_1(M),\mathbb{Z}/n\mathbb{Z})&=&{\rm Hom} ( {\rm H}_1(M) , \mathbb{Z} / n\mathbb{Z}) = {\rm Hom}( {\rm H}_1(M,\mathbb{Z}/n\mathbb{Z}) , \mathbb{Z}/ n \mathbb{Z}) \nonumber \\
                                                                         &=&{\rm H}^1(M,\mathbb{Z}/ n \mathbb{Z}).\nonumber 
\end{eqnarray}
For $\rho \in {\rm Hom}(\pi_1(M),\mathbb{Z}/n\mathbb{Z})$, by the definition of $j_1$,  we can identify $\rho^{*}_{M}({\mathrm id}) \in {\rm H}^1(M,\mathbb{Z}/ n \mathbb{Z})$ with $\rho$ under the identification above. We denote by $$\tilde{\rho} \in {\rm Hom}( {\rm H}_1(M,\mathbb{Z}/n\mathbb{Z}) , \mathbb{Z}/ n \mathbb{Z})$$ the corresponding homomorphism to $\rho$ and $\rho^{*}_{M}({\mathrm id})$. Now we move to state the main assertion in this subsection.
\begin{prop}
Notations being as above, let $u \in Z_2(M,\mathbb{Z}/ n \mathbb{Z})$ be a 2-cycle that represents $\Phi^1(\rho^{*}_{M}({\mathrm id})) \in {\rm H}_2(M,\mathbb{Z}/ n\mathbb{Z})$. Then there is a 2-chain $D \in C_2(M,\mathbb{Z})$ such that $D\ {\rm mod} \ n=u$ and there is a 1-cycle $\mathfrak{a} \in Z_1(M,\mathbb{Z})$ satisfying $\partial D = n \mathfrak{a}$. Let $[\mathfrak{a}]$ be the homology class in ${\rm H}_1(M,\mathbb{Z} / n\mathbb{Z})$ defined by $\mathfrak{a}$.  Then we have
$$ CS_c(\rho)=\tilde{\rho}([\mathfrak{a}]).$$
\end{prop}
\begin{proof}
The former assertion follows by examining the short exact sequence of the chain complexes,
$$0 \rightarrow C_i(M,\mathbb{Z}) \overset{\times n}{\rightarrow} C_i(M,\mathbb{Z}) \overset{{\rm mod} \ n}{\rightarrow} C_i(M,\mathbb{Z} / n\mathbb{Z}) \rightarrow 0,$$ For the latter assertion,  we  note that, by direct calculation, we can check $\Phi^2 \circ \beta^1 =\beta_2 \circ \Phi^1.$ Then, by Definition 5.1.2, we have,
\begin{eqnarray}
CS_c(\rho) &=& <\rho^{*}_{M}({\mathrm id}) \cup \beta^1(\rho^{*}_{M}({\mathrm id})), [M]> \nonumber \\
               &=& <\rho^{*}_{M}({\mathrm id}),\Phi^2(\beta^1(\rho^{*}_{M}({\mathrm id})))> \nonumber \\
               &=& \tilde{\rho}(\beta_2(\Phi^1(\rho^{*}_{M}({\mathrm id})))). \nonumber
\end{eqnarray} 
Next, we consider the following commutative diagram,
\begin{center}
\[
  \begin{CD}
     0 @>>>  \mathbb{Z}  @>{\times n}>>  \mathbb{Z}  @>>>  \mathbb{Z}/ n\mathbb{Z}  @>>>  0 \\
    @.     @VV{p_1}V  @VV{p_2}V  @VV{id}V   @. \\
     0 @>>>  \mathbb{Z}/n \mathbb{Z} @>{\times n}>>  \mathbb{Z}/n^2 \mathbb{Z} @>>>  \mathbb{Z}/n \mathbb{Z} @>>>  0,
  \end{CD}
\]
\end{center}
where $p_1$ and $p_2$ are   natural projections, and $id$ is the identity map. By considering the connecting homomorphism with respect to the singular homologies for each row, we see that $\beta_2 = {p_1}_* \circ \tilde{\beta}_2.$ Then the required statement immediately holds by the definition of $\tilde{\beta}_2$.
\end{proof}
 
\subsection{A topological analogue of the explicit formulas of the mod 2  Dijkgraaf-Witten invariants for double branched covers of the 3-sphere}
\ \ \ In this subsection, we prove  topological analogues of Theorem 4.2.2,  Corollary 4.2.3 and Corollary 4.2.4. Keeping the same notations as  Section 4.2,  Section 5.1 and Section5.2, we will consider the case $A=\mathbb{Z}/2\mathbb{Z}$ and $c={\mathrm id} \cup \beta({\mathrm id}) \in {\rm H}^3( A,  \mathbb{Z}/2 \mathbb{Z})$  in Definition 5.1.2. Let ${\cal L} = {\cal K}_1 \cup {\cal K}_2 \cup \cdots \cup {\cal K}_r$ be a tame link in the 3-sphere $S^3$. Let $h: M \rightarrow S^3$ be the double covering ramified over ${\cal L}$ obtained by the Fox completion ([F]) of the unramified covering $Y\rightarrow X  :=S^3 \backslash {\cal L}$ corresponding to the kernel of the surjective homomorphism ${\rm H}_1(X) \rightarrow \mathbb{Z}/2\mathbb{Z}$ that maps any meridian of ${\cal K}_i$ to $1\in \mathbb{Z}/2\mathbb{Z}.$ Then, by the topological analogue of  Gauss genus theory ([Mo2; Corollary]), there is an isomorphism
\begin{equation}
 g:{\rm H}_1(M) / 2 {\rm H}_1(M) \stackrel{\sim}{\rightarrow} T_{+},
\end{equation}
given by
$$ [\mathfrak{a}] \mapsto ({\rm lk}(h_*(\mathfrak{a}),{\cal K}_i) \ {\rm mod}\ 2), $$
where ${\rm lk}(\ ,\ ) \ {\rm mod}\ 2$ denotes the mod $2$  linking number.  Then we obtain the identifications
\begin{eqnarray}
{\rm Hom}(\pi_1(M),\mathbb{Z}/2\mathbb{Z})&=&{\rm Hom} ( {\rm H}_1(M) , \mathbb{Z} / 2\mathbb{Z}) = {\rm Hom}( {\rm H}_1(M,\mathbb{Z}/2\mathbb{Z}) , \mathbb{Z}/ 2 \mathbb{Z}) \nonumber \\
                                                         &=&{\rm Hom}(T_{+},\mathbb{Z} / 2\mathbb{Z}) \nonumber \\
                                                         &=&{\rm H}^1(M,\mathbb{Z}/ 2 \mathbb{Z}).\nonumber 
\end{eqnarray}
by Section 5.2. \par
Now we move to prove a topological  analogue of Corollary 4.2.3.
\begin{thm}Notations being as above, for $\rho \in {\rm Hom} (T_{+},\mathbb{Z}/2 \mathbb{Z})$,  we have
\begin{equation}
CS_c(\rho)=\sum_{i < j} \rho (e^{+}_{ij}){\rm lk}( {\cal K}_i,  {\cal K}_j)  \  {\rm mod}\ 2. \nonumber
\end{equation}
\end{thm}
\begin{proof}
We take a basis of $T_{+}$ over $\mathbb{Z}/2\mathbb{Z}$, $b_1,b_2,\cdots,b_{r-1}$,  which is defined by $$b_1=(1,0,0,\cdots,1),  b_2 =(0,1,0,0,\cdots,1),  \cdots, b_{r-1} = (0,0,\cdots,0,1,1).$$ We define $1 \leqq j_1<j_2\cdots<j_m \leqq r-1$ by
\begin{equation}
\rho(b_i)=
\begin{cases} 
1 & (i=j_1,j_2,\cdots,j_m)\\
0 & ({\rm otherwise}).\\
\end{cases} \nonumber
\end{equation}
\ \ \ We set  $\tilde{{\cal K}}_i := h^{-1}({\cal K}_i)$ and $\tilde{{\cal S}}_i := h^{-1}({\cal S}_i)$ for $i=1,2,\cdots,r$, where ${\cal S}_i$ is a Seifert surface of ${\cal K}_i$ in $S^3.$ Then, by following the construction of the isomorphism of the Poincar\'e duality, we can take a 2-cycle that represents $\Phi^1(\rho^{*}_{M}({\mathrm id})) \in {\rm H}_2(M,\mathbb{Z}/ 2\mathbb{Z})$, by $${\displaystyle \sum^{r}_{i=1} {\tilde{{\cal S}}_i}  - \sum_{i \in \{ j_1,j_2, \cdots ,j_m \} } \tilde{{\cal S}}_i }.$$ So we can take a 1-cycle $\mathfrak{a}_{\rho} \in Z_1(M,\mathbb{Z})$ in Proposition 5.2.1 by $$ \mathfrak{a}_{\rho}= {\displaystyle \sum^{r}_{i=1} \tilde{ {\cal K}_i } - \sum_{i \in \{ j_1,j_2, \cdots ,j_m \} } \tilde{ {\cal K}_i} }.$$ Therefore, by Proposition 5.1.2, we have
\begin{eqnarray}
CS_c(\rho) &=&\tilde{\rho}([\mathfrak{a}_{\rho}])  \nonumber \\
               &=&\rho(g(\mathfrak{a}_{\rho}))  \nonumber \\
               &=&\rho(({\rm lk}(h_*(\mathfrak{a}_{\rho}),{\cal K}_i)\ {\rm mod}\ 2 )  \nonumber \\
               &=&{\displaystyle \sum^{m}_{l=1} {\rm lk}(h_*(\mathfrak{a}_{\rho}),{\cal K}_{j_l}) \ {\rm mod}\ 2} \nonumber \\
               &=&{\displaystyle \sum^{m}_{l=1} \sum_{ i \notin \{ j_1,j_2,\cdots,j_m \} } {\rm lk}({\cal K}_{i},{\cal K}_{j_l}) \ {\rm mod}\ 2} \nonumber \\
               &=&{\displaystyle \sum_{i < j} \rho (e^{+}_{ij}) {\rm lk}( {\cal K}_i,  {\cal K}_j )} \ {\rm mod}\ 2. \nonumber
\end{eqnarray} 
\end{proof}
By Definition 5.1.2, the mod 2 Dijkgraaf-Witten invariant of $M$ associated to $c$ is given as follows.
\begin{coro}
Notations being as above,  we have
\begin{equation}
Z_c(M)={\displaystyle \sum_{\rho \in {\rm Hom}(T_+ , \mathbb{Z}/2\mathbb{Z})} {\mathrm exp} \left(\pi i { \underset{i < j}{\sum} \rho (e^{+}_{ij}) {\rm lk}( {\cal K}_i,  {\cal K}_j ) \ {\rm mod}\ 2}\right) } . \nonumber 
\end{equation}
\end{coro}

\begin{exm}
Let $L$ be a two-bridge link $B(a,b)$\ ($0<a<b,$ $b$: even, $(a,b)=1$). So we have $r=2$ and ${\rm Hom}(T_+,\mathbb{Z}/ 2\mathbb{Z})=\mathbb{Z}/ 2\mathbb{Z}$. Then, the two branched cover $M$ is the lens space $L(a,b)$. By Proposition 5.3.1 and ([Tu]), for $0 \neq \rho \in {\rm Hom}(T_+,\mathbb{Z}/ 2\mathbb{Z})$, we have
\begin{equation}
CS_c(\rho)={\displaystyle \sum_{k=1}^{b/2} (-1)^{[(2k-1)a/b]} \ {\rm mod}\ 2}, \nonumber
\end{equation}
where $[\ ]$ denotes the greatest integer function. Therefore, we also have 
\begin{equation}
Z_c(M)=
\begin{cases} 
2,& \ \mbox{if } {\displaystyle \sum_{k=1}^{b/2} (-1)^{[(2k-1)a/b]}\ {\rm is \  even,} }\\
0,& \ \  {\rm otherwise}.\\
\end{cases} \nonumber
\end{equation}
\end{exm}

\begin{rmk}
In [MOO], Murakami, Ohtsuki and Okada calculated the mod $n$ Dijkgraaf-Witten invariant for the 3-manifold obtained by the Dehn surgery on $S^3$ along  a framed link. Their formula expresses  the mod $n$ Dijkgraaf-Witten invariant in terms of the linking matrix of the framed link ([ibid; Proposition9.1]).
\end{rmk}

Hikaru Hirano\\
Faculty of Mathematics, Kyushu University\\
744, Motooka, Nishi-ku, Fukuoka, 819-0395, JAPAN\\
e-mail: ma218019@math.kyushu-u.ac.jp

\end{document}